\documentclass[11pt]{amsart}
\usepackage{amscd}
\usepackage[mathscr]{eucal}
\usepackage{amsfonts}
\usepackage{amsmath}
\usepackage{amsthm}
\usepackage{amssymb}
\usepackage{latexsym}
\usepackage{tabularx,array}
\newfont{\msam}{msam10}

\newtheorem{theorem}[]{Theorem}
\newtheorem{proposition}[]{Proposition}
\newtheorem{corollary}[]{Corollary}
\newtheorem{lemma}[]{Lemma}
\theoremstyle{definition}
\newtheorem{definition}[]{Definition}
\newtheorem{remark}[]{Remark}

\newtheorem{prop}[]{Proposition}
\newtheorem{cor}[]{Corollary}

\let\nc\newcommand

\nc{\la}{\label}

\def\bthm{\begin{theorem}}
\def\ethm{\end{theorem}}
\def\blemma{\begin{lemma}}
\def\elemma{\end{lemma}}
\def\bproof{\begin{proof}}
\def\eproof{\end{proof}}
\def\bprop{\begin{proposition}}
\def\eprop{\end{proposition}}
\def\bcor{\begin{corollary}}
\def\ecor{\end{corollary}}

\def\A{\mathbb{A}}

\nc{\Hom}{{\rm{Hom}}}
\nc{\Ext}{{\rm{Ext}}}
\nc{\HOM}{\underline{\rm{Hom}}}
\nc{\EXT}{\underline{\rm{Ext}}}
\nc{\TOR}{\underline{\rm{Tor}}}
\nc{\End}{{\rm{End}}}
\nc{\GL}{{\rm{GL}}}
\nc{\PGL}{{\rm{PGL}}}
\nc{\SL}{{\rm{SL}}}
\nc{\PSL}{{\rm{PSL}}}
\nc{\Rep}{{\rm{Rep}}}
\nc{\ad}{{\rm{ad}}}
\nc{\dlim}{\varinjlim}

\newcommand{\compl}{\mathbb C}

\newcommand{\rar}{\rightarrow}

\newcommand{\hb}{\hbar}

\begin{document}

\title{A variant of the Mukai pairing via deformation quantization.}

\author{Ajay C. Ramadoss}
\address{Departement Mathematik,
 ETH Z\"{u}rich, R\"{a}mistrasse 101, 8092 Z\"{u}rich}
\email{ajay.ramadoss@math.ethz.ch}
\maketitle

\begin{abstract}
We give a new method to prove a formula computing a variant of Caldararu's Mukai pairing \cite{Cal1}. Our method is based on some important results in the area of deformation quantization. In particular, part of the work of Kashiwara and Schapira in \cite{KS} as well as an algebraic index theorem of Bressler, Nest and Tsygan in \cite{BNT},\cite{BNT1} and \cite{BNT2} are used. It is hoped that our method is useful for generalization to settings involving certain singular varieties.
\end{abstract}

\section{Introduction.}
Let $X$ denote a smooth proper complex variety (we remind the reader that $X$ has the Zariski topology). We denote the corresponding (compact) complex manifold by $X^{an}$. In \cite{Cal1}, A. Caldararu defined a Mukai pairing $\langle -,- \rangle_M$ on $\text{HH}_{\bullet}(X)$, the Hochschild homology of $X$. On the other hand, one has the Hochschild-Kostant-Rosenberg (HKR) isomorphism $I_{HKR}:\text{HH}_{\bullet}(X) \rar \oplus_i \text{H}^{i-\bullet}(X,\Omega^i_X)$. It was implicitly proven in \cite{Mar1} (and explicitly so in \cite{Ram1} following \cite{Mar1}) that
$$\langle a,b \rangle_M = \int_X I_{HKR}(b) \wedge J(I_{HKR}(a)) \wedge \text{td}(T_X)$$
where $J$ is the involution multiplying an element of $\text{H}^{i-\bullet}(X,\Omega^i_X)$ by ${(-1)}^i$. This result has recently been of interest: applications of this and related results appear, for instance, in \cite{HMS},\cite{MaS} and \cite{Ram4}. A closely related pairing was $\langle -,-\rangle_{Shk}$ was constructed in \cite{Ram3} following D. Shkyarov in \cite{Shkl}. It turns out that the latter pairing is directly related to a natural definition of Fourier-Mukai transforms in Hochschild homology (see \cite{Shkl} and \cite{Ram3}). This definition of Fourier-Mukai transforms in Hochschild homology is equivalent to an earlier, but less direct definition in \cite{Cal1} (also see Section 4.3 of\cite{KS}). A careful comparison between this pairing and Caldararu' Mukai pairing was performed in \cite{Ram3} to show that \begin{theorem} \la{main} \begin{equation} \la{pair} \langle a,b \rangle_{Shk} = \int_X I_{HKR}(a) \wedge I_{HKR}(b) \wedge \text{td}(T_X) \text{.} \end{equation} \end{theorem} In these notes, we provide a different proof of this result based on the work of Kashiwara-Schapira \cite{KS} and an algebraic index theorem of Bressler, Nest and Tsygan in \cite{BNT},\cite{BNT1} and \cite{BNT2} (the latter being a very important result in the general area of deformation quantization). Unlike the earlier approach from  \cite{Mar1}, \cite{Ram1} and \cite{Ram3} (also see \cite{Ram5} for further details), this approach requires that we work over $\mathbb C$. However, it gives a clear connection (hitherto missing) between the computation of a ``Mukai pairing'' and a large body of work in deformation quantization, algebraic index theorems and related topics. We also point out that essentially the same result has been proven in \cite{Griv} using what we use from \cite{KS} and a deformation to the normal cone argument. While the (interesting) approach in \cite{Griv} is far more concise than the one via \cite{Mar1}, \cite{Ram1} and \cite{Ram3}, the argument there is geometric and not intrinsic to $X$. Readers with some background in deformation quantization and algebraic index theory would also find the approach in this note far more concise than the earlier one (that in \cite{Mar1}, \cite{Ram1} and \cite{Ram3}), while remaining algebraic and intrinsic to $X$ in nature. Further, unlike the earlier approach, this method is likely to lend itself to generalization to more general settings involving certain singular varieties. We also remark that as far as (the above cited as well as other) recent applications are concerned, a formula for $\langle -,-\rangle_{Shk}$ is as useful/suitable as one for  $\langle -,- \rangle_M$.

\textbf{Acknowledgements.} I am very grateful to Damien Calaque, Giovanni Felder, Pierre Schapira and Xiang Tang for some very useful discussions. In fact, this note is motivated by a joint project with Xiang Tang. This work is partially supported by the Swiss National Science Foundation under the project ``Topological quantum mechanics and index theorems" (Ambizione Beitrag Nr.$\text{ PZ}00\text{P}2\_127427/1$).

\section{Preliminaries.}

Let $\omega_X:=\Omega^n_X[n]$. Let $\Delta: X \rar X \times X$ denote the diagonal embedding. Recall from \cite{KS} that one has the following commutative diagram in the bounded derived category $\text{D}^{b}(\mathcal O_X)$ of coherent sheaves on $X$.
$$\begin{CD}
\Delta^*\Delta_*\mathcal O_X @>>\text{td}> \Delta^!\Delta_*\omega_X\\
@VV{I_{HKR}}V                      @AA{\widehat{I_{HKR}}}A\\
\oplus_i \Omega^i_X[i] @>>{\tau}> \oplus_i \Omega^i_X[i]
\end{CD}$$
Let $\text{D}$ denote the map on hypercohomology induced by $\text{td}:\Delta^*\Delta_*\mathcal O_X \rar \Delta^!\Delta_*\omega_X$. Let $I_{HKR}$,$\widehat{I_{HKR}}$ and $\tau$ continue to denote the maps induced on hypercohomology by  $I_{HKR}$,$\widehat{I_{HKR}}$ and $\tau$ respectively. Applying hypercohomologies, one obtains the following commutative diagram.
\begin{equation} \label{CD1}
\begin{CD}
\mathbb{H}^{-\bullet}(X,\Delta^*\Delta_*\mathcal O_X) @>>\text{D}> \mathbb{H}^{-\bullet}(X,\Delta^!\Delta_*\omega_X)\\
@VV{I_{HKR}}V                      @AA{\widehat{I_{HKR}}}A\\
\oplus_i \text{H}^{i-\bullet}(X,\Omega^i_X)  @>>{\tau}>  \oplus_i \text{H}^{i-\bullet}(X,\Omega^i_X)
\end{CD}
\end{equation}
The map $\widehat{I_{HKR}}$ has been constructed in \cite{KS}, Section 5.4\footnote{A similar map has been constructed in Section 1 of \cite{Ram1}.}. M. Kashiwara and P.Schapira show us in \cite{KS}\footnote{We remark that all constructions/results in Chapter 5 of \cite{KS}, which are done in the setting of complex manifolds, work in the algebraic setting that we are working in.} that

 \begin{prop} Theorem \ref{main} is equivalent to the assertion that the map $\tau$ in \eqref{CD1} is the wedge product with $\text{Td}(TX)$.
  \end{prop}
  \begin{proof}
  Let $X,Y$ be smooth projective varieties over $\compl$. Recall that any $\Phi \in \text{D}^{b}_{coh}(X \times Y)$ gives an integral transform $\Phi_*^{cal}:\text{HH}_{\bullet}(X) \rar \text{HH}_{\bullet}(Y)$ (see Section 4.3 of \cite{Cal1}). On hypercohomologies, Corollary 4.2.2 of \cite{KS} yields a pairing
  $$\langle -,- \rangle_{KS}: \text{HH}_{\bullet}(X) \otimes \text{HH}_{\bullet}(X) \rar \compl \text{.}$$
We remark that $\text{HH}_{\bullet}(X)$ is also the hypercohomology of the complex of Hochschild chains of $\mathcal O_X^{op}$, which is equal to $\text{HH}_{\bullet}(X)$ since $\mathcal O_X^{op}=\mathcal O_X$. In particular, we are not making this identification via the duality map described at the end of Section 4.1 of \cite{KS}. Lemma 4.3.4 of \cite{KS} then tells us that after identifying $HH_{\bullet}(X \times Y)$ with $HH_{\bullet}(Y) \otimes HH_{\bullet}(X)$,\footnote{$X \times Y$ is viewed as $Y \times X$ while making this identification.}
  \begin{equation} \la{KSpair} \Phi_*^{cal}(\alpha)=\langle \text{Ch}(\Phi),\alpha \rangle_{KS} \text{ . }\end{equation}
   Let $\Phi =\mathcal O_{\Delta}$ ($\Delta$ here denoting the diagonal in $X \times X$). In this case, $\Phi_*^{cal}=\text{id}$ (see Section 5 of \cite{Cal1}). Then, by Theorem 5 of \cite{Ram3}\footnote{Note that we are not using any part of \cite{Ram3} that depends on the Mukai pairing formula computed in \cite{Mar1} and \cite{Ram1}.}, $\text{Ch}(\Phi)=\sum_i e_i \otimes f_i$ where the $e_i$ and $f_j$ are homogenous bases of $\text{HH}_{\bullet}(X)$ such that $\langle f_j,e_i \rangle_{Shk} =\delta_{ij}$. On the other hand, equation \eqref{KSpair} applied to $\alpha=e_i$ tells us that $\langle f_j, e_i\rangle_{KS} =\delta_{ij}$, thus showing that
   $\langle -,- \rangle_{KS} = \langle -,- \rangle_{Shk}$. Finally, the end of Section 5.4 of \cite{KS} shows us that
   $$\langle a,b \rangle_{KS}= \int_X I_{HKR}(a) \wedge \tau(I_{HKR}(b)) \text{.}$$
   \end{proof}
  We therefore, need to show that $\tau=(-\wedge \text{Td}(TX))$. In our method, the following proposition from \cite{KS}, Chapter 5 is the first step in this direction.
\begin{prop} \la{one}
(i) $\Delta^*\Delta_*\mathcal O_X$ is a ring object in $\text{D}^{b}(\mathcal O_X)$, and $ \Delta^!\Delta_*\omega_X$ is a left module object over $\Delta^*\Delta_*\mathcal O_X$ in $\text{D}^{b}(\mathcal O_X)$.\\
(ii) Further, $\text{td}$ is a morphism of left $\Delta^*\Delta_*\mathcal O_X$ modules in $\text{D}^{b}(\mathcal O_X)$.
.\end{prop}
\begin{proof}
The ring structure of $\Delta^*\Delta_*\mathcal O_X$ in $\text{D}^{b}(\mathcal O_X)$ is given by the composite map
$$ \Delta^*\Delta_*\mathcal O_X \otimes^{\mathbb L}_{\mathcal O_X} \Delta^*\Delta_*\mathcal O_X \cong \Delta^*(\Delta_*\mathcal O_X \otimes^{\mathbb L}_{\mathcal O_{X \times X}} \Delta_*\mathcal O_X) \stackrel{\Delta^*\mu}{\longrightarrow} \Delta^*\Delta_*\mathcal O_X $$
where $\mu$ is induced by the product map $\Delta_* \mathcal O_X \otimes_{\mathcal O_{X \times X}} \Delta_*\mathcal O_X$.\\
The module structure of $\Delta^!\Delta_*\omega_X$ over $\Delta^*\Delta_*\mathcal O_X$ is realized via the composite map
$$\Delta^*\Delta_*\mathcal O_X \otimes^{\mathbb L}_{\mathcal O_X} \Delta^!\Delta_*\omega_X \cong \Delta^! (\Delta_*\mathcal O_X \otimes^{\mathbb L}_{\mathcal O_{X \times X}} \Delta_*\omega_X) \stackrel{\Delta^!\nu}{\longrightarrow} \Delta^!\Delta_*\omega_X \text{.}$$
Here, $\nu$ is the composite map
$$ \Delta_*\mathcal O_X \otimes^{\mathbb L}_{\mathcal O_{X \times X}} \Delta_*\omega_X \cong \Delta_* (\Delta^*\Delta_*\mathcal O_X \otimes_{\mathcal O_X} \omega_X) \rar \Delta_*\omega_X $$
the last arrow being induced by the adjunction $\Delta^*\Delta_*\mathcal O_X \rar \mathcal O_X$.\\
The morphism $\text{td}$ was constructed in \cite{KS} as follows.
$$\Delta^*\Delta_*\mathcal O_X \cong \mathcal O_X \otimes^{\mathbb L}_{\mathcal O_X} \Delta^*\Delta_*\mathcal O_X \cong \Delta^!(\mathcal O_X \boxtimes \omega_X) \otimes^{\mathbb L}_{\mathcal O_X} \Delta^*\Delta_*\mathcal O_X \cong \Delta^!((\mathcal O_X \boxtimes \omega_X) \otimes^{\mathbb L}_{\mathcal O_{X \times X}} \Delta_*\mathcal O_X)$$ $$ \Delta^!((\mathcal O_X \boxtimes \omega_X) \otimes^{\mathbb L}_{\mathcal O_{X \times X}} \Delta_*\mathcal O_X) \cong \Delta^!\Delta_*\omega_X $$
That $\text{td}$ is a morphism of left $\Delta^*\Delta_*\mathcal O_X$-modules is more or less a direct consequence of the fact that $\otimes^{\mathbb L}_{\mathcal O_X}$ is associative.
\end{proof}

\begin{cor}
For all $\alpha \in\oplus_i \text{H}^{i-\bullet}(X,\Omega^i_X)$, $\tau(\alpha)=\alpha \wedge \tau(1)$.
\end{cor}

\begin{proof}
The ring structure of $\Delta^*\Delta_* \mathcal{O}_X$ induces a product $\bullet$ on $\mathbb{H}^{-\bullet}(X,\Delta^*\Delta_*\mathcal O_X)$. By Proposition \ref{one},
$$\text{D}(a \bullet b) =a \bullet \text{D}(b) $$
for all $a,b \in \mathbb{H}^{-\bullet}(X,\Delta^*\Delta_*\mathcal O_X)$. It follows from Lemma 5.4.7 of \cite{KS} that for all $a,b \in \mathbb{H}^{-\bullet}(X,\Delta^*\Delta_*\mathcal O_X)$,
$$ \widehat{I_{HKR}}(I_{HKR}(a) \wedge \beta)= a \bullet \widehat{I_{HKR}}(\beta) $$
The desired corollary now follows from the fact that $I_{HKR}$ and $\widehat{I_{HKR}}$ are isomorphisms.
\end{proof}

Recall that for any $E \in \text{D}^{b}(\mathcal O_X)$, one has the {\it Chern character} $\text{ch}(E) \in \mathbb{H}^{0}(X,\Delta^*\Delta_*\mathcal O_X)$. By Theorem 4.5 of \cite{Cal2}, $I_{HKR}(\text{ch}(E))$ is the Chern character of $E$ in the classical sense. The {\it Euler class} $\text{eu}(E)$ is defined as the element
 $\widehat{I_{HKR}}^{-1}(\text{D}(\text{ch}((E)))$ of $\oplus_i \text{H}^{i}(X,\Omega^i_X)$.
Note $\tau(1)=\text{eu}(\mathcal O_X)$. In order to compute the $\langle-,-\rangle_{Shk}$, we therefore, need to show that
$$ \text{eu}(\mathcal O_X)=\text{Td}(T_X) \text{.} $$

Before we proceed, let us make a clarification. Recall that $\Delta^*\Delta_*\mathcal O_X$ is represented in the derived category $\text{D}^{-}(\mathcal O_X)$ of bounded above complexes of quasi-coherent sheaves on $X$ by the complex of $\widehat{\mathcal{C}_{\bullet}}(\mathcal O_X)$ of completed Hochschild chains (after turning it into a cochain complex by inverting degrees). Recall from \cite{Y} that $\widehat{\mathcal{C}_{n}}(\mathcal O_X):= \varprojlim_k \frac{\mathcal O_X^{\otimes n+1}}{I_n^k}$ where $I_n$ is the kernel of the product map $\mathcal O_X^{\otimes n+1} \rar \mathcal O_X$. Let $\mathcal C_{\bullet}(\mathcal O_X)$ be the complex of sheaves of $X$ associated to the complex of presheaves $U \mapsto \text{C}_{\bullet}(\Gamma(U,\mathcal O_X))$ (the Hochschild chain complex here being the naive algebraic one). One similarly defines $\mathcal{C}_{\bullet}^{red}(\mathcal O_X)$ using reduced Hochschild chains. There are natural maps $\mathcal{C}_{\bullet}^{red}(\mathcal O_X) \leftarrow \mathcal C_{\bullet}(\mathcal O_X) \rar \widehat{\mathcal{C}_{\bullet}}(\mathcal O_X)$ of complexes of sheaves on $X$ which are quasiisomorphisms. In the following section, when thinking of the complex of Hochschild chains on $X$, we shall be thinking of $\mathcal C_{\bullet}^{red}(\mathcal O_X)$ (which has the same hypercohomology as $\widehat{\mathcal{C}_{\bullet}}(\mathcal O_X)$).

\section{The Euler class of $\mathcal O_X$.}

It remains to show that $\text{eu}(\mathcal O_X)=\text{Td}(T_X)$. The original intrinsic computation for this from \cite{Mar1} (see \cite{Ram1} for details) is very lengthy and involved. Further, its connections to deformation quantization and related areas are not clear. Another, more recent proof due to \cite{Griv} uses deformation to the normal cone. We now sketch our new approach to this question. Let $\mathcal D_X$ denote the sheaf of (algebraic) differential operators on $X$. Recall that the Hochschild-Kostant-Rosenberg quasiisomorphism on Hochschild chains induces an isomorphism $I_{HKR}:\text{HC}_0^{per}(\mathcal O_X) \rightarrow \prod_{p=-\infty}^{\infty}\text{H}^{2p}(X^{an},\compl)$. On the other hand, a construction very similar to the trace density construction of Engeli-Felder on Hochschild chains induces an isomorphism $\chi:\text{HC}_0^{per}(\mathcal D_{X^{an}}) \rightarrow \prod_{p=-\infty}^{\infty}\text{H}^{2n-2p}(X^{an},\compl)$ (see \cite{EnFe}, \cite{PPT} and \cite{Will}). Further, one has a natural map $(-)^{an}: \text{HC}_0^{per}(\mathcal D_X) \rar \text{HC}_0^{per}(\mathcal D_{X^{an}})$\footnote{Indeed, if $f:X^{an} \rar X$ is the canonical map, one has a natural map $f^{-1}(\mathcal{CC}_{\bullet}^{per}(\mathcal D_X)) \rar \mathcal{CC}_{\bullet}^{per}(\mathcal D_{X^{an}})$ of complexes of sheaves on $X^{an}$, and hence in the derived category $\text{D}(\text{Sh}_{\compl}(X^{an}))$ of sheaves of $\compl$-vector spaces on $X^{an}$. By adjunction, one gets a natural map $\mathcal{CC}_{\bullet}^{per}(\mathcal D_X) \rar Rf_*(\mathcal{CC}_{\bullet}^{per}(\mathcal D_{X^{an}}))$, to which we apply $R\Gamma(X,-)$. $Rf_*$ and $R\Gamma$ are extend to $\text{D}(\text{Sh}_{\compl}(X^{an}))$ and $\text{D}(\text{Sh}_{\compl}(X))$ respectively since $f_*$ and $\Gamma(X,-)$ have finite cohomological dimension.}. The natural homomorphism  $\mathcal O_X \rar \mathcal D_X$ of sheaves of algebras on $X$ induces maps on Hochschild as well as negative cyclic and periodic cyclic homologies. These maps shall be denoted by $\iota$.  The following proposition is closely related to a Theorem in \cite{BNT} (also see \cite{BNT1} and \cite{BNT2}).

\begin{prop} \la{two}
The following diagram commutes.
$$\begin{CD}
\text{HC}_0^{per}(\mathcal O_X) @>>(-)^{an} \circ \iota> \text{HC}_0^{per}(\mathcal D_{X^{an}})\\
@VV{I_{HKR}}V   @V{\chi}VV\\
\prod_{p=-\infty}^{\infty}\text{H}^{2p}(X^{an},\compl) @>(-\wedge \text{Td}(T_X))>> \prod_{p=-\infty}^{\infty}\text{H}^{2n-2p}(X^{an},\compl)
\end{CD}$$
\end{prop}

Note that for any sheaf of algebras $\mathcal A$ on $X$, one has natural maps $\text{HC}_0^{-}(\mathcal A) \rar  \text{HC}_0^{per}(\mathcal A)$ and $\text{HC}_0^{-}(\mathcal A) \rar \text{HH}_0(\mathcal A)$. Also recall that one has a natural projection $\text{H}^{2p}(X^{an},\compl) \rar \text{H}^{p,p}(X^{an},\compl)$ for all $p$. We omit the proof of the following proposition.

\begin{prop} \la{three}
The following diagrams commute.\\
(a)
$$\begin{CD}
\text{HC}_0^{-}(\mathcal O_X) @>{I_{HKR}}>> \prod_{p=-\infty}^{\infty}\text{H}^{2p}(X^{an},\compl)\\
@VVV        @VVV\\
\text{HH}_0(\mathcal O_X) @>{I_{HKR}}>> \oplus_p \text{H}^{p,p}(X^{an},\compl)
\end{CD}$$
(b)
$$\begin{CD}
\text{HC}_0^{-}(\mathcal D_{X^{an}}) @>\chi>> \prod_{p=-\infty}^{\infty}\text{H}^{2n-2p}(X^{an},\compl)\\
@VVV   @VVV\\
\text{HH}_0(\mathcal D_{X^{an}}) @>\chi>> \text{H}^{2n}(X^{an},\compl)
\end{CD}$$
(c)
$$\begin{CD}
\text{HC}_0^{-}(\mathcal O_X) @>(-)^{an} \circ \iota>> \text{HC}_0^{-}(\mathcal D_{X^{an}})\\
@VVV    @VVV\\
\text{HH}_0(\mathcal O_X) @>(-)^{an} \circ \iota>> \text{HH}_0(\mathcal D_{X^{an}})
\end{CD}$$
\end{prop}

\begin{prop} \la{four}
The following diagram commutes.
$$\begin{CD}
\text{HH}_0(\mathcal O_X) @>(-)^{an} \circ \iota>> \text{HH}_0(\mathcal D_{X^{an}})\\
@VV{I_{HKR}}V  @V{\chi}VV\\
\oplus_p \text{H}^{p,p}(X^{an},\compl) @>>{(-\wedge \text{Td}(T_X))_{2n}}> \text{H}^{2n}(X^{an},\compl)
\end{CD}$$
\end{prop}

\begin{proof}
We note that the natural map $\text{HC}^{-}_{0}(\mathcal O_X) \rar \text{HH}_0(\mathcal O_X)$ is surjective. Indeed, after applying $I_{HKR}$, we are reduced to verifying that $\text{H}^{p}(X,\text{Ker}(d:\Omega^p_X \rar \Omega^{p+1}_X) \rar \text{H}^p(X,\Omega^p_X)$ is surjective. By Serre's GAGA, it suffices to verify that $\text{H}^{p}(X^{an},\text{Ker}(d:\Omega^p_{X^{an}} \rar \Omega^{p+1}_{X^{an}}) \rar \text{H}^p(X^{an},\Omega^p_{X^{an}})$ is surjective. This follows from the fact that any closed $(p,p)$-form defines an element of $\text{H}^{p}(X^{an},\text{Ker}(d:\Omega^p_{X^{an}} \rar \Omega^{p+1}_{X^{an}})$ as well.

 Hence, any $y \in \text{HH}_0(\mathcal O_X)$ lifts to an element $\tilde{y} \in \text{HC}^{-}_{0}(\mathcal O_X)$. For notational brevity, we denote $\chi \circ (-)^{an}$ by $\chi$ for the rest of this proof. Now, $\chi \circ \iota(y)= (\chi \circ \iota(\tilde{y}))_{2n}$ by Proposition \ref{three}, parts (b) and (c). Further, $(\chi \circ \iota(\tilde{y}))_{2n}= (I_{HKR}(\tilde{y}) \wedge \text{Td}(T_X))_{2n}$ by Proposition \ref{two}. Finally, $(I_{HKR}(\tilde{y}) \wedge \text{Td}(T_X))_{2n}=(I_{HKR}(y) \wedge \text{Td}(T_X))_{2n}$ by Proposition 3, part (a) and the fact that $\text{Td}(T_X) \in \oplus_p \text{H}^{p,p}(X^{an},\compl)$.\\
\end{proof}

The following proposition is a crucial point in this note.

\begin{prop} \la{five}
The following diagram commutes.
$$\begin{CD}
\text{HH}_0(\mathcal O_X) @>D>> {\mathbb H}^0(X,\Delta^!\Delta_*\omega_X)\\
@VV{(-)^{an} \circ \iota}V        @VV{(\widehat{I_{HKR}}^{-1}(-))_{2n}}V\\
\text{HH}_0(\mathcal D_X) @>\chi>>  \text{H}^{2n}(X^{an},\compl)
\end{CD}$$

\end{prop}

\begin{proof}
Let $\pi:X \rar pt$ be the natural projection. The object $\mathcal O_X$ of $\text{Perf}(\mathcal O_{X \times pt})$ induces a morphism $\pi_*: \text{Perf}(\mathcal O_X) \rar \text{Perf}(pt)$ in the homotopy category $\text{Ho}(dg-cat)$ of DG-categories modulo quasiequivalences (see Section 8 of \cite{T}). The notation $\pi_*$ is justified by the fact that the functor from $\text{D}(\text{Perf}(X))$ to $\text{D}(\text{Perf}(pt))$ induced by $\pi_*$ is indeed the derved pushforward $\pi_*$. This induces a map $\pi_*:\text{HH}_0(\mathcal O_X) \rar \text{HH}_0(\mathcal O_{pt})=\compl$ wich coincides with the pushforward on Hochschild homologies from \cite{KS} (see Theorem 5 of \cite{Ram3}). On the other hand, one has $\pi_*:\oplus_p \text{H}^{p,p}(X^{an},\compl) \rar \text{H}^0(pt,\compl)=\compl$, which coincides with $\int_{X^{an}}$. By the proof of Proposition 5.2.3 of \cite{KS}, ${\widehat{I_{HKR}}}^{-1} \circ D$ commutes with $\pi_*$. On the other hand, let $\text{Perf}(\mathcal D_X)$ denote the DG-category of perfect complexes of (right) $\mathcal D_X$-modules that are quasi-coherent as $\mathcal O_X$-modules. One has a map $\pi^{\mathcal D}_*: \text{Perf}(\mathcal D_X) \rar \text{Perf}(pt)$ in $\text{Ho}(dg-cat)$. The functor incuced by $\pi^{\mathcal D}_*$ on derived categories  maps $M \in \text{D}(\text{Perf}(\mathcal D_X))$ to $\pi_*(M^{an} \otimes^{\mathbb L}_{\mathcal D_{X^{an}}} \mathcal O_{X^{an}})$\footnote{The latter is indeed in $ \text{D}(\text{Perf}(pt))$: see \cite{ScS} for instance.}.
By Section 8 of \cite{T}, $\pi^{\mathcal D}_*$ induces a map $\pi^{\mathcal D}_*: \text{HH}_0(\text{Perf}(\mathcal D_X)) \rar \text{HH}_0(pt) \cong \compl$ on Hochschild homologies. By Proposition \ref{pTT} at the end of this section, the composite map
\begin{equation} \la{dperf} \text{HH}_{\bullet}(\text{Perf}(\mathcal D_X)) \rar \text{HH}_{\bullet}(\mathcal D_X) \stackrel{(-)^{an}}{\rar} \text{HH}_{\bullet}(\mathcal D_{X^{an}}) \end{equation} is an isomorphism (the first map in the above composition is the trace map from Section 4 of \cite{K}).  $\pi^{\mathcal D}_*$ therefore, induces a $\compl$-linear functional on $\text{HH}_{0}(\mathcal D_{X^{an}})$, which we shall continue to denote by $\pi^{\mathcal D}_*$. It follows from \cite{EnFe} and \cite{Ram2} that
$$\pi^{\mathcal D}_*=\int_{X^{an}} \circ \chi: \text{HH}_0(\mathcal D_{X^{an}}) \rar \compl \text{.}$$
Since $\int_{X^{an}}:\text{H}^{2n}(X^{an},\compl) \rar \compl$ is an isomorphism, the required proposition follows once we check that $\pi_*=\pi^{\mathcal D}_* \circ (-)^{an} \circ \iota$. This follows from the fact that the diagram
$$\begin{CD}
\text{HH}_0(\mathcal O_X) @>(-)^{an} \circ \iota>> \text{HH}_0(\mathcal D_X)\\
@AAA   @AAA\\
\text{HH}_0(\text{Perf}(\mathcal O_X)) @>(-) \otimes_{\mathcal O_X} {\mathcal D_X} >> \text{HH}_0(\text{Perf}(\mathcal D_X))
\end{CD}$$
(the left vertical arrows being the trace isomorphism from Section 4 of \cite{K} and the right vertical arrow being the composite map \eqref{dperf}) commutes as well as the observation that for $E \in \text{D}(\text{Perf}(\mathcal O_X))$,
$$ \pi^{\mathcal D}_* \iota (E)=\pi_*((E \otimes_{\mathcal O_X} \mathcal D_X)^{an} \otimes^{\mathbb L}_{\mathcal D_{X^{an}}} \mathcal O_{X^{an}})= \pi_*E^{an}$$ (recall that $\pi_*E=\pi_*E^{an}$ in $\text{D}(\text{Perf}(pt))$ by Serre's GAGA).
\end{proof}

By Propositions \ref{four} and \ref{five}
$$ (I_{HKR}(\alpha) \wedge \tau(1))_{2n}=(I_{HKR}(\alpha) \wedge \text{Td}(T_X))_{2n} $$
for all $\alpha \in \text{HH}_0(\mathcal O_X)$. Hence, $\text{eu}(\mathcal O_X)=\tau(1)=\text{Td}(T_X)$. To complete the proof of Proposition 5, we sketch the proof of the following proposition.

\begin{prop} \la{pTT}
$\text{HH}_{\bullet}(\text{Perf}(\mathcal D_X)) \cong \text{HH}_{\bullet}(\mathcal D_{X^{an}})$. This isomorphism is realized by the composite map \eqref{dperf}.
\end{prop}

\begin{proof}
One has to verify that the arguments of B. Keller in Section 5 of \cite{K} go through when $\mathcal O_X$ is replaced by $\mathcal D_X$. The crucial part here is the analog of Theorem 5.5 of \cite{K} (originally proven as Propositions 5.2.2-5.2.4 of \cite{TT}) when $\mathcal O_X$ is replaced by $\mathcal D_X$. This is done in Propositions 3.3.1-3.3.3 of \cite{DY} (which prove the analog of Theorem 5.5 of \cite{K} in a much more general setting: in particular, when $\mathcal O_X$ is replaced by $\mathcal R_X$ where $\mathcal R_X$ is a sheaf of quasicoherent $\mathcal O_X$-algebras (possibly noncommutative)). Let $Y$ be any quasi-compact, quasi-separated scheme over $\compl$ with $V$,$W$ quasi-compact open subschemes of $Y$ such that $Y=V \cup W$. Following the arguments of Sections 5.6 and 5.7 of \cite{K}, one obtains a morphism of Mayer-Vietoris sequences
$$\begin{CD}
\text{HH}_i(\text{Perf}(\mathcal D_Y)) @>>> \text{HH}_i(\text{Perf}(\mathcal D_V)) \oplus \text{HH}_i(\text{Perf}(\mathcal D_W)) @>>> \text{HH}_i(\text{Perf}(\mathcal D_{V \cap W})) @>>> \text{HH}_{i-1}(\text{Perf}(\mathcal D_Y))\\
@VVV @VVV @VVV @VVV\\
\text{HH}_i(\mathcal D_{Y^{an}}) @>>> \text{HH}_i(\mathcal D_{V^{an}}) \oplus \text{HH}_i(\mathcal D_{W^{an}}) @>>> \text{HH}_i(\mathcal D_{(V \cap W)^{an}}) @>>> \text{HH}_{i-1}(\mathcal D_{Y^{an}})
\end{CD}$$
(for each $i \in \mathbb Z$). The vertical arrows in the above diagram are induced by the composite map \eqref{dperf}. As in Section 5.9 of \cite{K}, we may then reduce the proof of the desired proposition to proving the desired proposition when $X$ is affine with trivial tangent bundle. For the rest of this proof, we assume that this is indeed the case.

Since $\mathcal D_X-\text{mod}$ denote the Abelian category of (right) $\mathcal D_X$-modules are quasi-coherent $\mathcal O_X$-modules. There is an equivalence of abelian categories between $D_X-\text{mod}$ and $\mathcal D_X-\text{mod}$, where  $D_X:=\Gamma(X,\mathcal D_X)$ (see \cite{DY}, example 1.1.5). Hence, one has an equivalence of DG-categories between $\text{Perf}(D_X)$ and $\text{Perf}(\mathcal D_X)$ (this follows, for instance, from Lemma 2.2.1 of \cite{DY}). This equivalence induces an isomorphsm \\ $\text{HH}_{\bullet}(\text{Perf}(D_X)) \stackrel{\cong}{\rar} \text{HH}_{\bullet}(\text{Perf}(\mathcal D_X))$. Further, there is a natural map $\text{HH}_{\bullet}(D_X) \rar \text{HH}_{\bullet}(\mathcal D_X)$ such that the following diagram commutes.
$$\begin{CD}
\text{HH}_{\bullet}(\text{Perf}(D_X)) @>\cong>>   \text{HH}_{\bullet}(\text{Perf}(\mathcal D_X))\\
@VV{\cong}V    @VVV\\
\text{HH}_{\bullet}(D_X) @>>> \text{HH}_{\bullet}(\mathcal D_X)
\end{CD}$$
In the above diagram, the vertical arrows are trace maps from Section 4 of \cite{K}. For honest algebras, they yield isomorphisms. We are therefore, reduced to verifying that the composite map
\begin{equation} \la{qismf} \text{HH}_{\bullet}(D_X) \rar \text{HH}_{\bullet}(\mathcal D_X) \stackrel{(-)^{an}}{\rar} \text{HH}_{\bullet}(\mathcal D_{X^{an}}) \end{equation}
is an isomorphism. Let $\mathcal D^{\bullet}_{X^{an}}$ denote the Dolbeault resolution of the sheaf $\mathcal D_{X^{an}}$. This is a sheaf of DG-algebras on $X$. Let $\text{C}_{\bullet}(\mathcal D^{\bullet}_{X^{an}})$ denote the complex of global sections of the complex of completed Hochschild chains on $X$ (see \cite{Ram2}, Section 3.3). There is a natural map of complexes $\text{C}_{\bullet}(D_X) \rar \text{C}_{\bullet}(\mathcal D^{\bullet}_{X^{an}})$ inducing \eqref{qismf} on homology. To prove that this is a quasi-isomorphism, we filter algebraic and holomorphic differential operators by order and consider the induced map on the $E^2$-terms of the spectral sequences from Section 3.3 of \cite{Bryl}. This turns out to be induced on homology by the natural map from the algebraic De-Rham complex $(\Omega^{2n-\bullet}(T^*X),d_{DR}^{alg})$ to the Dolbeault complex $(\Gamma(X^{an},\Omega^{2n-\bullet}_{T^*X^{an}} \otimes_{\mathcal O_{X^{an}}} \Omega^{0,\bullet}_{X^{an}}),d+\bar{\partial})$ \footnote{Here, $\Omega^{\bullet}_{T^*X^{an}}$ is the complex of sheaves on $X^{an}$ whose sections on each open subset $U$ of $X^{an}$ are holomorphic forms on $T^*U$ that are algebraic along the fibres of the projection $T^*U \rar U$. $d$ is the (holomorphic) De-Rham differential on this complex.}. That this is a quasiisomorphism amounts to the assertion that natural map from the algebraic De-Rham complex of $X$ to the smooth De-Rham complex of $X^{an}$ is a quasiisomorphism (see \cite{Groth}).

\end{proof}

\section{A proof of Proposition \ref{two}.}
One notes that the following diagram commutes.
$$\begin{CD}
\text{HH}_0(\mathcal O_X) @>(-)^{an}>> \text{HH}_0(\mathcal O_{X^{an}})\\
@VV{\iota}V                             @V{\iota}VV\\
\text{HH}_0(\mathcal D_X) @>(-)^{an}>> \text{HH}_0(\mathcal D_{X^{an}})
\end{CD}$$
To prove Proposition \ref{two}, it therefore, suffices to show that the following diagram commutes (where $Y:=X^{an}$).
\begin{equation} \la{p2anal}
\begin{CD}
\text{HC}_0^{per}(\mathcal O_Y) @>>\iota> \text{HC}_0^{per}(\mathcal D_{Y})\\
@VV{I_{HKR}}V   @V{\chi}VV\\
\prod_{p=-\infty}^{\infty}\text{H}^{2p}(Y,\compl) @>(-\wedge \text{Td}(T_Y))>> \prod_{p=-\infty}^{\infty}\text{H}^{2n-2p}(Y,\compl)
\end{CD}
\end{equation}
In other words, we now work with a complex manifold rather than an algebraic variety. Recall that there is a deformation quantization $\A^{\hb}_{T^*Y}$ of $\mathcal O_{T^*Y}[[\hb]]$ such that $\pi^{-1}\mathcal D_Y \hookrightarrow \A^{\hb}_{T^*Y}[\hb^{-1}]$ and $\A^{\hb}_{T^*Y}[\hb^{-1}]$ is flat over $\pi^{-1}\mathcal D_Y $. Here, $\pi:T^*Y \rar Y$ is the canonical projection.

In this situation, one has a natural map $\pi^{-1}:\text{HC}_0^{per}(\mathcal D_Y) \rar \text{HC}_0^{per}(\A^{\hb}_{T^*Y}[\hb^{-1}])$. Indeed, if $\mathcal{U}:=\{U_i\}$ is a good open cover of $Y$, one has a natural map of complexes between the periodic cyclic-Cech complex $\text{C}^{\vee}(\mathcal{U}, \mathcal{CC}_{\bullet}^{per}(\mathcal D_Y))$ and  $\text{C}^{\vee}(\mathcal{V}, \mathcal{CC}_{\bullet}^{per}(\A^{\hb}_{T^*Y}[\hb^{-1}]))$ where  $\mathcal{V}:=\{\pi^{-1}(U_i)\}$. Similarly, one has a natural map $\pi^{-1}:\text{HC}_0^{per}(\mathcal O_Y) \rar \text{HC}_0^{per}(\A^{\hb}_{T^*Y})$. Further, one has a trace density map $\chi_{FFS}:\text{HC}_0^{per}(\A^{\hb}_{T^*Y}[\hb^{-1}]) \rar\prod_p \text{H}^{2n-2p}(T^*Y,\compl)((\hb))$ (see \cite{BNT},\cite{EnFe},\cite{FFS},\cite{Will}). Note that we can compose $\chi_{FFS}$ with the natural map $\beta: \text{HC}_0^{per}(\A^{\hb}_{T^*Y}) \rar \text{HC}_0^{per}(\A^{\hb}_{T^*Y}[\hb^{-1}])$. We shall abuse notation to denote $\chi_{FFS} \circ \beta$ by $\chi_{FFS}$. Let $i:Y \rar T^*Y$ denote inclusion as the zero section. The following proposition is clear.

\begin{prop} \la{six}
The diagram
$$\begin{CD}
\text{HC}_0^{per}(\mathcal D_Y) @>\pi^{-1}>> \text{HC}_0^{per}(\A^{\hb}_{T^*Y}[\hb^{-1}])\\
@VV{\chi}V       @V{\chi_{FFS}}VV\\
\prod_p \text{H}^{2n-2p}(Y,\compl)((\hb)) @>\pi^*>> \prod_p \text{H}^{2n-2p}(T^*Y,\compl)((\hb))
\end{CD}$$
 commutes. Further, $i^* \circ \pi^* =\text{id}$ on $\prod_p \text{H}^{2n-2p}(Y,\compl)((\hb))$.
\end{prop}

One has a ``principal symbol" homomorphism $\sigma:\A^{\hb}_{T^*Y} \rar \mathcal O_{T^*Y}$.
The following theorem is from \cite{BNT} (see also \cite{BNT1} and \cite{BNT2}).

\begin{theorem} \la{t1}
The following diagram commutes.
$$\begin{CD}
\text{HC}_0^{per}(\A^{\hb}_{T^*Y}) @>\sigma>> \text{HC}_0^{per}(\mathcal O_{T^*Y})\\
@VV{\chi_{FFS}}V                       @V{I_{HKR}}VV\\
\prod_p \text{H}^{2n-2p}(T^*Y,\compl)((\hb)) @<(-) \cup \pi^*\text{Td}(T_Y) << \prod_p \text{H}^{2p}(T^*Y,\compl)((\hb))
\end{CD}$$
\end{theorem}

\begin{prop} \la{seven}
The following diagrams commute.
$$\begin{CD}
\text{HC}_0^{per}(\mathcal O_Y) @>\pi^{-1}>> \text{HC}_0^{per}(\A^{\hb}_{T^*Y})\\
@VV{\pi^*}V  @V{\text{id}}VV\\
\text{HC}_0^{per}(\mathcal O_{T^*Y}) @<\sigma<< \text{HC}_0^{per}(\A^{\hb}_{T^*Y})
\end{CD}$$

$$\begin{CD}
\text{HC}_0^{per}(\mathcal O_Y) @>\iota>> \text{HC}_0^{per}(\mathcal D_Y)\\
@VV{\pi^{-1}}V  @V{\pi^{-1}}VV\\
\text{HC}_0^{per}(\A^{\hb}_{T^*Y}) @>\beta>> \text{HC}_0^{per}(\A^{\hb}_{T^*Y}[\hb^{-1}])
\end{CD}$$

$$\begin{CD}
\text{HC}_0^{per}(\mathcal O_Y) @>\pi^*>> \text{HC}_0^{per}(\mathcal O_{T^*Y})\\
@VV{I_{HKR}}V  @V{I_{HKR}}VV\\
\prod_p \text{H}^{2p}(Y,\compl) @>\pi^*>> \prod_p \text{H}^{2p}(T^*Y,\compl)((\hb))
\end{CD}$$
\end{prop}

Denote the bottom arrow in the diagram of equation \eqref{p2anal} by $\theta$ (after extending scalars to $\compl((\hb))$ in the codomain).
Since $I_{HKR}:\text{HC}_0^{per}(\mathcal O_Y)  \rar \prod_p \text{H}^{2p}(Y,\compl)$ is an isomorphism, Propositions \ref{six},\ref{seven} and Theorem \ref{t1} together imply that
$$ \theta(\alpha)=i^*(\pi^*(\alpha) \cup \pi^*(\text{Td}(T_Y)))=i^* \pi^* (\alpha \cup \text{Td}(T_Y))=\alpha \cup \text{Td}(T_Y) \text{.}$$
This proves Proposition \ref{two}.

\end{document}